\documentclass{article}
\usepackage{amssymb,amsmath,amsthm,graphicx,kotex} 
\usepackage{scalerel}

\def\qed{\hfill {\hbox{${\vcenter{\vbox{               
   \hrule height 0.4pt\hbox{\vrule width 0.4pt height 6pt
   \kern5pt\vrule width 0.4pt}\hrule height 0.4pt}}}$}}}

\def\tr{\ast}

\newtheorem{theorem}{Theorem}

\newtheorem{proposition}[theorem]{Proposition}

\theoremstyle{definition}
\newtheorem{example}{Example}
\newtheorem{definition}{Definition}
\newtheorem{remark}{Remark}

\date{}

\title{\Large \textbf{Marked Graph Mosaics}} 

\author{
Seonmi Choi\footnote{Email: smchoi@knu.ac.kr. Partially supported by Basic Science Research Program through the National Research Foundation of Korea(NRF) funded by the Ministry of Education(2021R1I1A1A01049100) and the National Research Foundation of Korea (NRF) grant funded by the Korean government (MSIT) (No. 2022R1A5A1033624). }
 \and
Sam Nelson\footnote{Email: Sam.Nelson@cmc.edu. Partially supported by
Simons Foundation Collaboration Grant 702597.}}

\begin{document}
\maketitle

\begin{abstract}
We consider the notion of mosaic diagrams for surface-links using marked
graph diagrams. We establish bounds, in some cases tight, on the mosaic
numbers for the surface-links with ch-index up to 10. As an application,
we use mosaic diagrams to enhance the kei counting invariant for unoriented 
surface-links as well as classical knots and links.
\end{abstract}

\parbox{6in} {\textsc{Keywords:} Mosaic knots, Surface-links, Marked graph diagrams, kei homset enhancements

\smallskip

}

\section{\large\textbf{Introduction}}\label{I}

\textit{Surface-links} are compact surfaces smoothly embedded in $\mathbb{R}^4$ or $S^4$, i.e. surfaces which are knotted and linked in 4-space. Surface-links include many more distinct topological types of unknotted objects -- spheres, tori, projective planes, Klein bottles, etc. -- compared with classical knots, and additionally include both orientable and non-orientable cases.

Introduced in \cite{Lomonaco}, \textit{marked graph diagrams} are knot diagrams
with \textit{marked vertices} representing saddle points of a \textit{surface-link}.
A marked graph diagram satisfying certain mild conditions determines
a surface-link up to ambient isotopy in $\mathbb{R}^4$, and marked graph
diagrams together with the \textit{Yoshikawa moves} provide a convenient 
diagrammatic calculus for combinatorial computation with surface-links. Moreover, marked graph diagrams and their Yoshikawa equivalence classes provide a diagrammatic way to represent cobordisms between classical knots and links.

A \textit{mosaic diagram} for a classical knot $K$ is a rectangular (usually 
square)
arrangement of square tiles containing crossings, arcs or nothing 
such that the arcs join to form a diagram of $K$. Mosaics were used in 
\cite{LK} to define \textit{quantum knots}, elements of Hilbert spaces 
generated by mosaic diagrams.

In this paper we take the first steps toward extending these constructions to 
the case of surface-links by considering mosaic presentations 
for surface-links using marked graph diagrams. We establish a set of tiles 
and Yoshikawa moves
for marked graph mosaics and provide mosaic diagrams for each of the 
surface-links in the Yoshikawa table of surface-links with up to ch-index 10,
establishing an upper bound on mosaic number for these surface-links.
As an application we use
mosaic presentations to define a new enhancement of the kei counting invariant
for classical knots and links as well as for surface-links. As with mosaic 
number, we can compute an upper bound with respect to a certain
ordering on the new enhancement from a given diagram of a surface-link or
classical knot or link.

The paper is organized as follows. In Section \ref{P} we review some 
preliminaries about knot mosaics and marked graph diagrams. In Section 
\ref{MGM} we introduce marked graph mosaics and obtain some results including
upper bounds, some tight, on the the mosaic numbers of both orientable and 
non-orientable surface-links with ch-index less than or equal to 10. 
In Section \ref{K} we define kei-colored mosaics and use them to enhance the
kei counting invariant for classical knots and links as well as surface-links.
We conclude in Section \ref{Q} with some questions for future research.

\section{\large\textbf{Preliminaries}}\label{P}

We review knot mosaics and recall surface-links, marked graph diagrams and their relationships. 

\subsection{Surface-links and marked graph diagrams}

A \textit{surface-link} is the image of a closed surface smoothly (piecewise 
linear and locally flatly) embedded in $\mathbb{R}^4$ (or $S^4$). If it is 
called a \textit{surface-knot}, then the underlying surface is connected. 
A surface-link is \textit{orientable} if the underlying surface is orientable; 
otherwise, it is \textit{nonorientable} or \textit{unorientable}. 
An \textit{unoriented} surface-link is either an unorientable surface-link or
an orientable surface link without a specified orientation.
Two surface-links $F$ and $F'$ are \textit{equivalent} if there exists an 
orientation-preserving homeomorphism 
$h : \mathbb{R}^4 \rightarrow \mathbb{R}^4$ such that $h(F)=F'$.
There are many useful schemes for describing for surface-links since it is 
difficult to directly deal with surface-links in 4-space for research. For 
example, broken surface diagrams, marked graph diagrams, motion pictures etc. 
See \cite{CarterKamadaSaito,Kamadabook,Kamadabook2,Yoshikawa} for more 
information. 

We use an effective tool for handling surface-links known as a \textit{marked 
graph diagram}. A \textit{marked graph} is a spatial graph embedded in 
$\mathbb{R}^3$ possibly with $4$-valent vertices decorated by a line segment 
like \scalerel*{\includegraphics{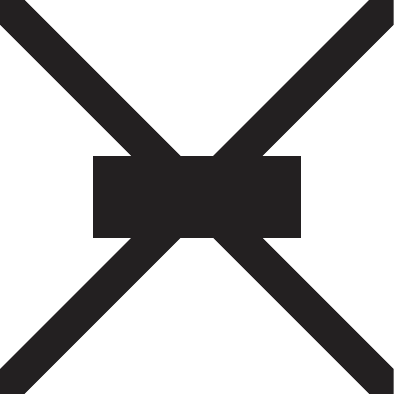}}{B}.
We call such a line segment a \textit{marker} and a vertex with a marker a 
\textit{marked vertex}.

An orientation of edges incident with a marked vertex is one of two types of 
the orientation, such as \scalerel*{\includegraphics{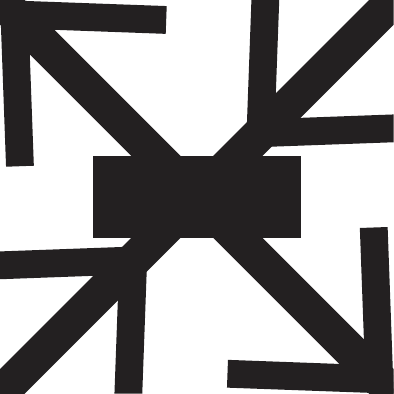}}{B} 
or \scalerel*{\includegraphics{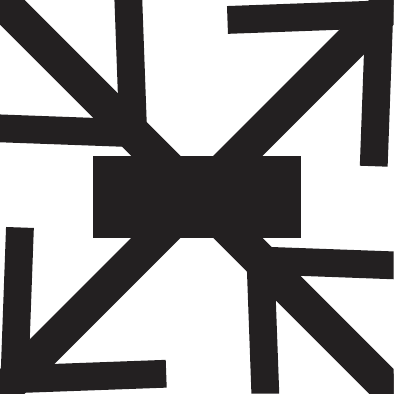}}{B}.
A marked graph is said to be \textit{orientable} if it admits an orientation.
Otherwise, it is \textit{non-orientable}. 
Two (oriented) marked graphs are said to be \textit{equivalent} if they are 
ambient isotopic in $\mathbb{R}^3$ keeping the rectangular neighborhoods 
and markers (with orientation).
In the same way as a link diagram, one can define a \textit{marked graph 
diagram} which is a diagram in $\mathbb{R}^2$ with classical crossings and 
marked vertices.

For each marked vertex \scalerel*{\includegraphics{MarkedVertex1.pdf}}{B} of 
a marked graph diagram $D$, the local diagram obtained by splicing in a 
direction consistent with its marker (say $+$ direction), looks like 
\scalerel*{\includegraphics{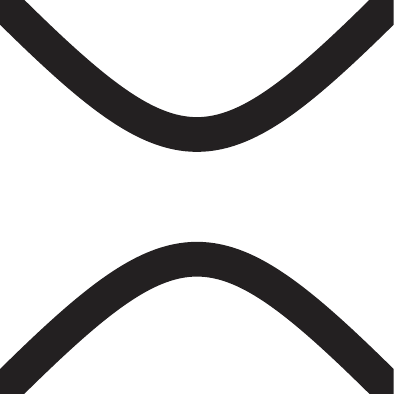}}{B}. 
By applying this in the opposite direction (called $-$ direction), the resulting local diagram looks like \scalerel*{\includegraphics{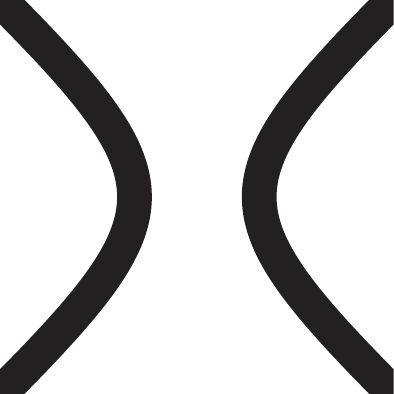}}{B}. 
Therefore one can obtain two classical link diagrams, denoted by $L_{+}(D)$ and $L_{-}(D)$, from $D$ by splicing every marked vertices in $+$ direction and $-$ direction, respectively. 
We call $L_{+}(D)$ and $L_{-}(D)$ the \textit{positive} and \textit{negative resolutions} of $D$, respectively. 


A marked graph diagram $D$ is said to be \textit{admissible} if both resolutions $L_{-}(D)$ and $L_{+}(D)$ are trivial. A marked graph is said to be \textit{admissible} if it has an admissible marked graph diagram.
For example, it is easy to check that a marked graph diagram $D$ of the spun trefoil as follows is admissible.
\[\scalebox{0.3}{\includegraphics{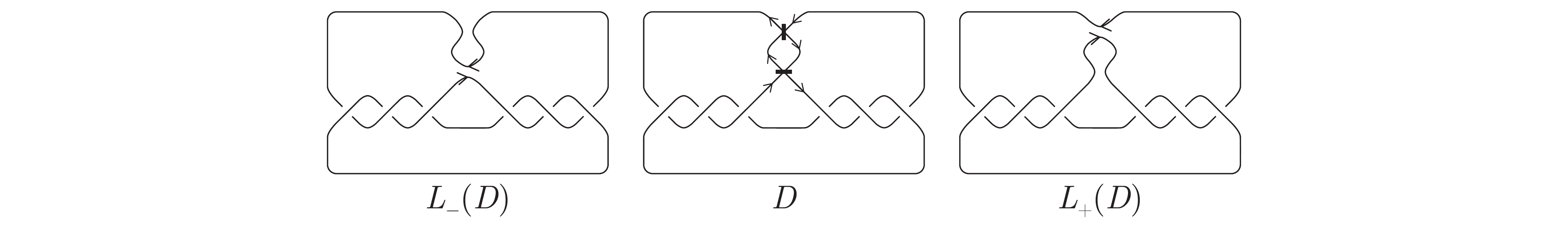}}\]

Let $D$ be a admissible marked graph diagram. Then
a surface-link $F(D)$ can be constructed and it is uniquely determined from $D$ 
up to equivalence. 
Conversely, every surface-link $F$ can be expressed by an admissible marked graph diagram $D$, that is, $F(D)$ is equivalent to $F$.
See \cite{KawauchiShibuyaSuzuki,Lomonaco,Yoshikawa} for more details. 
For example, the correspondence between the marked graph diagram and the standard projective plane are illustrated in the following figure.
\[\scalebox{0.3}{\includegraphics{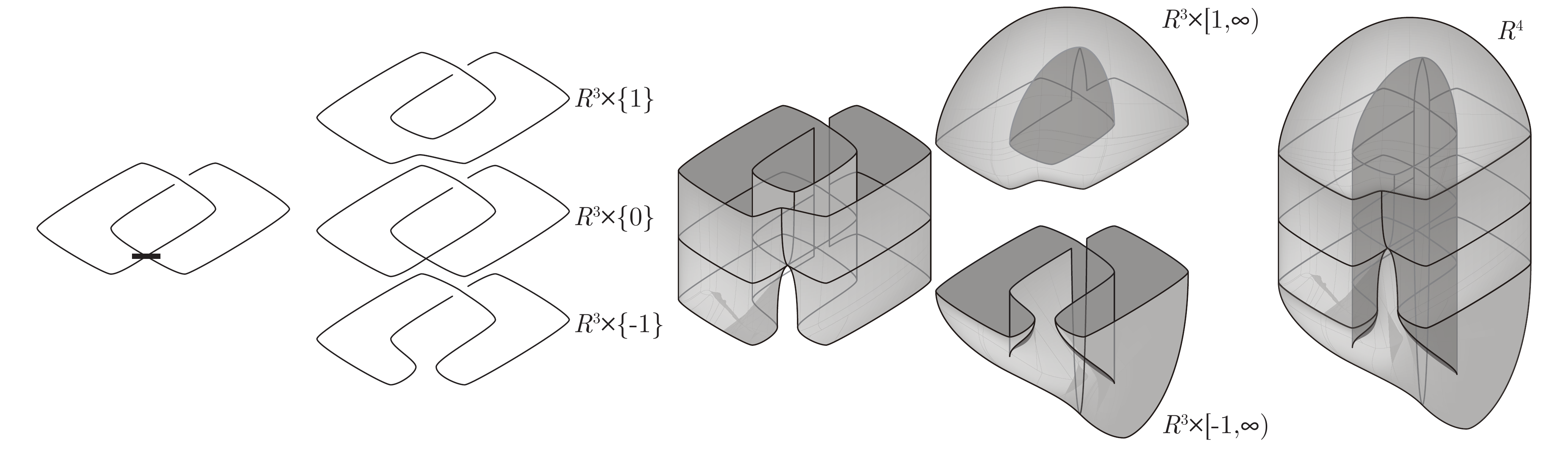}}\]

The equivalence moves $\Gamma_{1}, \cdots, \Gamma_{8}$ for marked graph diagrams is called \textit{ Yoshikawa moves} \cite{Yoshikawa}.
\[\scalebox{0.2}{\includegraphics{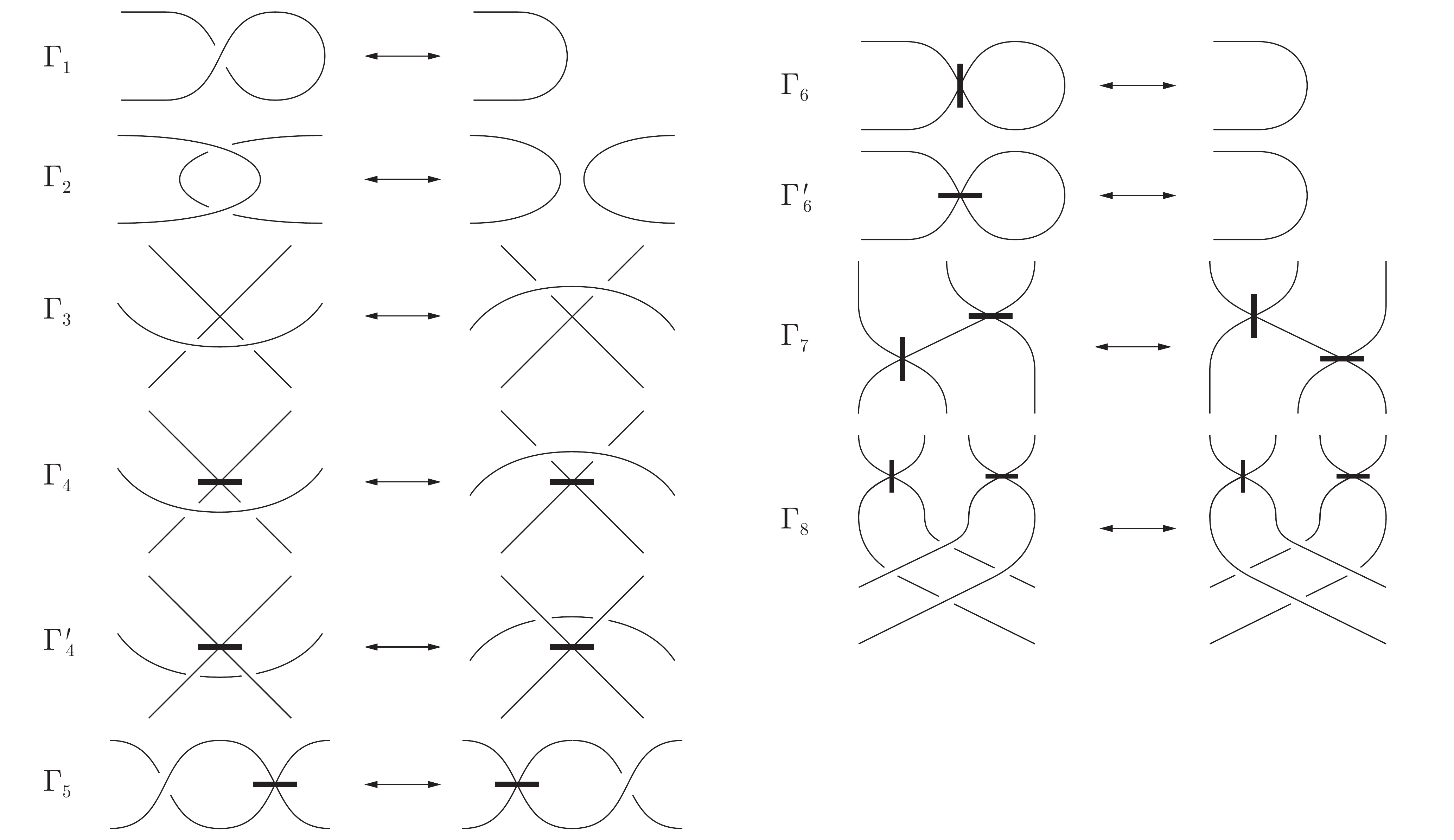}}\]


\begin{proposition}[\cite{KeartonKurlin,Swenton,Yoshikawa}]
Two marked graph diagrams $D$ and $D'$ present equivalent oriented surface-links if and only if $D$ can be obtained from $D'$ by a finite sequence of ambient isotopies in $\mathbb{R}^{2}$ and Yoshikawa moves.
\end{proposition}

\begin{definition}
Let $K$ be a marked graph diagram. The \textit{ch-index} of $K$, denoted 
$\mathrm{ch}(K)$, is the total number of crossings and marked vertices in $K$.
\end{definition}

\subsection{\large\textbf{Mosaic Knots}}\label{MK}

A \textit{mosaic (unoriented) tile} is one of rectangles with arcs and possibly with one crossing, depicted as follows.
\[\includegraphics{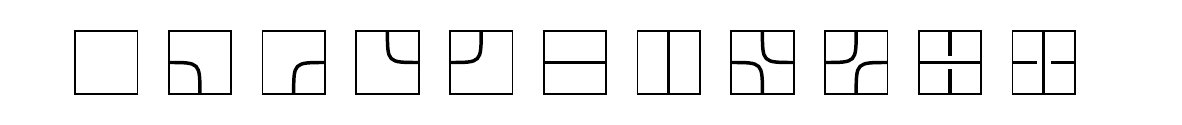}\]
The set of mosaic tiles $T_{0}, T_{1}, \cdots, T_{10}$ is denoted by $\mathbb{T}^{(u)}$ and
there are exactly $5$ tiles, up to rotation.
The endpoints of an arc on a mosaic tile are called \textit{connection points} of the tile and are also located the center of an edge. 
There are tiles with $0$, $2$ and $4$ connection points in $\mathbb{T}^{(u)}$. 
\[\includegraphics{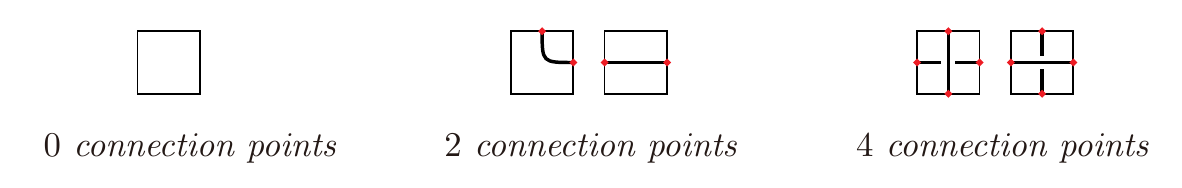}\]

An \textit{$(m, n)$-mosaic} is an $m\times n$ matrix whose entries are mosaic tiles in $\mathbb{T}^{(u)}$.
If $m=n$, then it is simply called an \textit{$n$-mosaic}.
The sets of $(m, n)$-mosaics and $n$-mosaics are denoted by $\mathbb{M}^{(m, n)}$ and $\mathbb{M}^{(n)}$, respectively.
\[\includegraphics{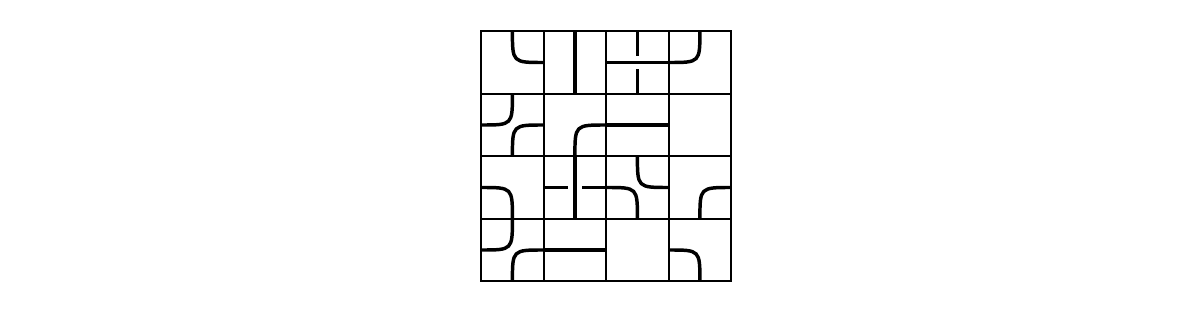}\]
Two tiles in a mosaic are said to be \textit{contiguous} if they lie immediately next to each other in the same either row or column.
A tile in a mosaic is said to be \textit{suitably connected} if all its connection points touch the connection points of contiguous tiles. 
all its connection points meet the connection points of contiguous tiles.
Note that for $4$-mosaic illustrated above, its $(2, 2)$-entry tile is suitably connected, but its $(3,3)$-entry tile is not suitably connected.

\begin{definition}
A \textit{knot $(m, n)$-mosaic} is an $(m, n)$-mosaic in which all tiles are suitably connected. The set of all knot $(m, n)$-mosaic is the subset of $\mathbb{M}^{(m, n)}$, denoted by $\mathbb{K}^{(m, n)}$. 
If $m=n$, then it is called a \textit{knot $n$-mosaic} and its set is denoted by $\mathbb{K}^{(n)}$.
\end{definition}

\begin{example}
The trefoil $3_{1}$ has a knot $5$-mosaic and $4$-mosaic, as follows.
\[\includegraphics{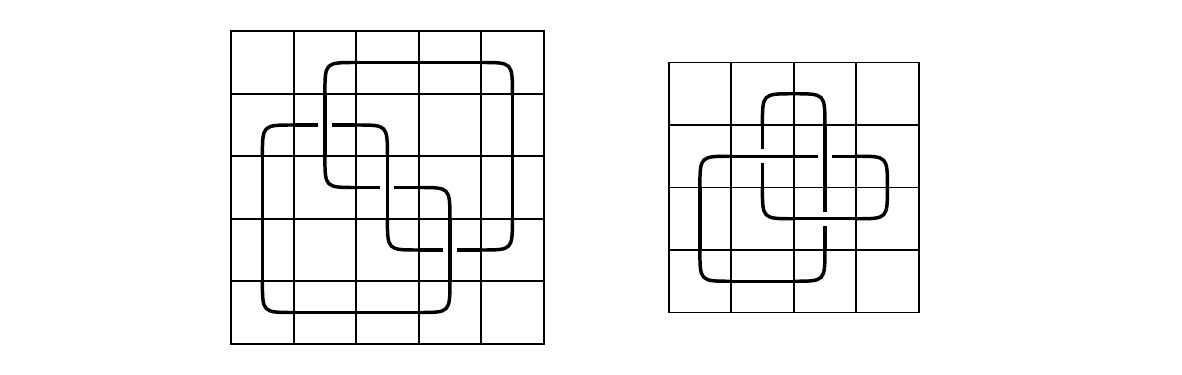}\]
\end{example}

For the equivalence for mosaic knots, there are planar isotopy moves and Reidemeister moves by using mosaic tiles. The non-deterministic tiles are necessary to define the moves, as follows :
\[\includegraphics{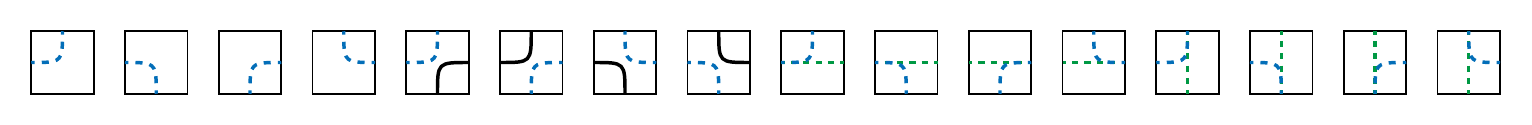}\]
Each non-deterministic tile means two types of tiles. 
\[\includegraphics{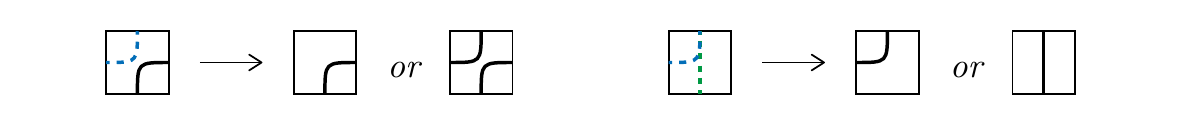}\]
Non-deterministic tiles labeled by the same letter $A$ or $B$ are synchronized.
\[\includegraphics{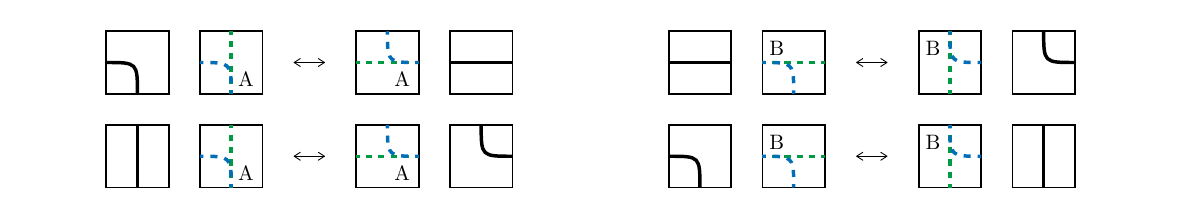}\]

The equivalence of mosaic knots consists of 11 moves for planar isotopy, 2 moves for Reidemeister moves I, 4 moves for Reidemeister moves II and 6 moves for Reidemeister moves III.
\begin{itemize}
  \item[0.] Planar isotopy moves : $11$ types
  \[\scalebox{0.8}{\includegraphics{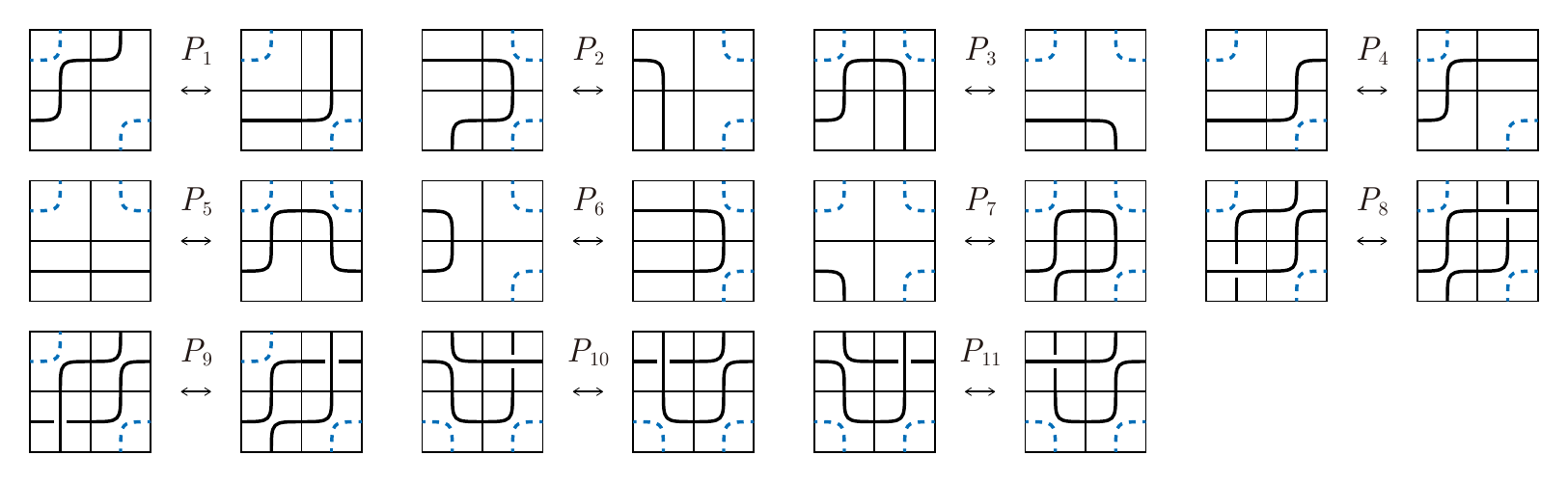}}\]
  \item[1.] Reidemeister moves I : $2$ types
  \[\scalebox{0.8}{\includegraphics{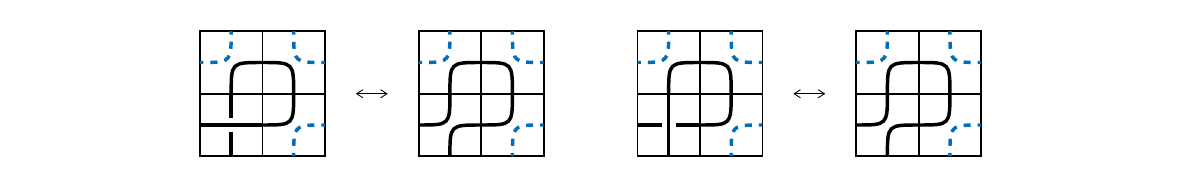}}\]
  \item[2.] Reidemeister moves II : $4$ types
  \[\scalebox{0.8}{\includegraphics{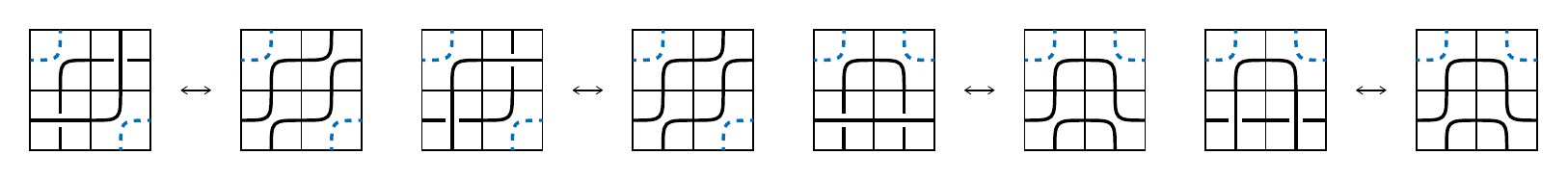}}\]
  \item[3.] Reidemeister moves III : $6$ types
  \[\scalebox{0.8}{\includegraphics{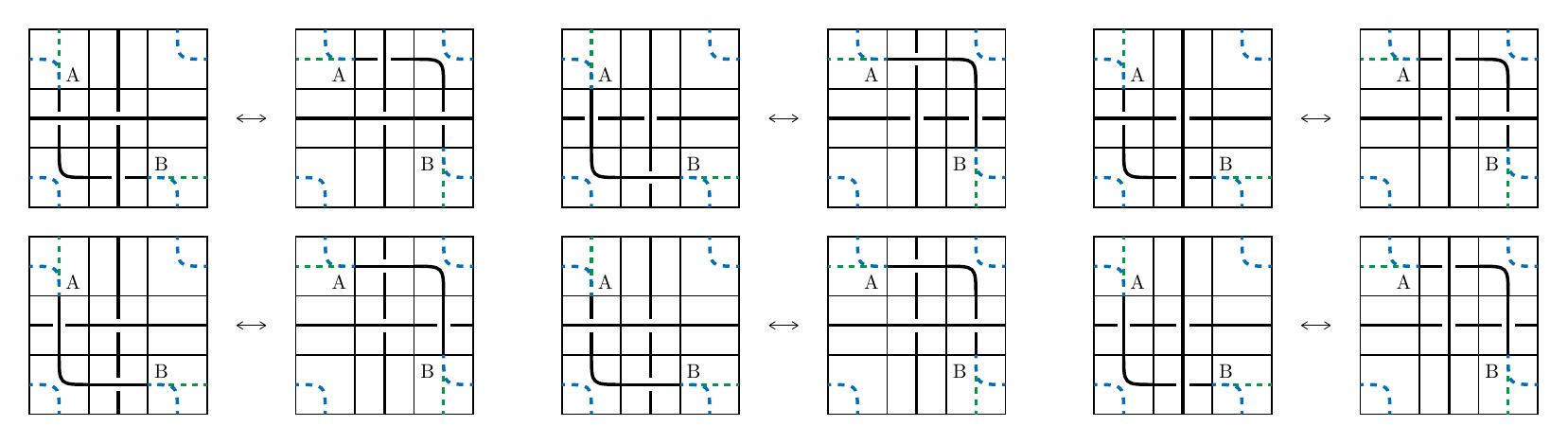}}\]
\end{itemize}

All mosaic moves are permutations on the set $\mathbb{M}^{(n)}$ of $n$-mosaics. Indeed, they are also in the group of all permutations of the set $\mathbb{K}^{(n)}$ of knot $n$-mosaics.

\begin{definition}
The \textit{ambient isotopy group $\mathbb{A}(n)$} is the subgroup of the group of all permutations of the set $\mathbb{K}^{(n)}$ generated by all planar isotopy moves and all Reidemeister moves.  
\end{definition}

Two $n$-mosaics $M$ and $M'$ are said to be \textit{of the same knot $n$-type}, denoted by 
  $M \overset{n}{\sim} M',$
  if there exists an element of $\mathbb{A}(n)$ such that it transforms $M$ into $M'$.
Two $n$-mosaics $M$ and $M'$ are said to be \textit{of the same knot type} 
  if there exists a non-negative integer $k$ such that 
  $$i^{k}M \overset{n+k}{\sim} i^{k}M',$$
  where $i : \mathbb{M}^{(j)} \rightarrow \mathbb{M}^{(j+1)}$ is the mosaic injection by adding a row and a column consisting of only empty tiles.
  
In \cite{LK}, Lomonaco and Kauffman conjectured that tame knot theory is equivalent to knot mosaic theory and in \cite{KS}, Kuriya and Shehab proved the conjecture. 

\begin{proposition}
Let $K$ and $K'$ be two knot mosaics of two tame knots $k$ and $k'$, respectively. 
Then $K$ and $K'$ are of the same knot mosaic type if and only if $k$ and $k'$ are of the same knot type.
\end{proposition}

\begin{definition}
The \textit{mosaic number} of a knot (or a link) $K$, denoted by $m(K)$, is the smallest integer $n$ for which $K$ can be represented by a $n$-mosaic.
\end{definition} 

It is obvious that the mosaic number is an invariant for knots and links. For example, the mosaic number of $3_1$ is $4$ and it is easy to show this. In the papers \cite{OHLL, LLPP}, they calculated the mosaic number of knots up to $8$ crossings.

\section{\large\textbf{Marked Graph Mosaics}}\label{MGM}

Let $\mathbb{T}^{(u)}_{M}$ denote the set of $2$ Symbols, called \textit{mosaic (unoriented) tiles with markers}, as follows : 
\[\includegraphics{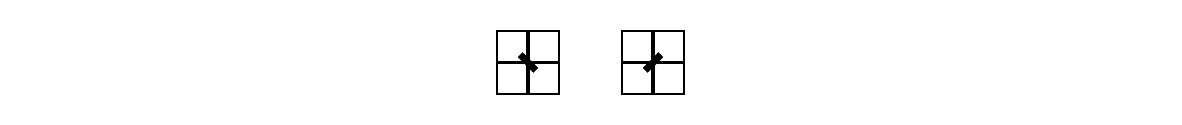}\]

Note that the two tiles are the same up to rotation and have $4$ connection points. 
For constructing an $n$-mosaic for marked graph diagrams, consider all tiles of $\mathbb{T}^{(u)}\cup \mathbb{T}^{(u)}_{M}$ as elementary tiles.
\[\includegraphics{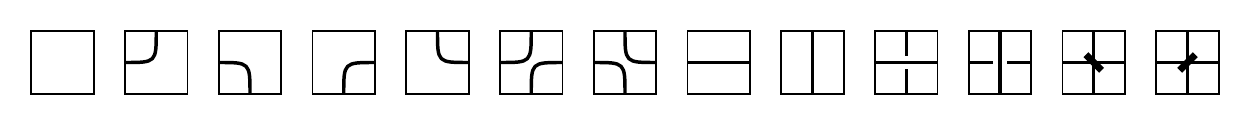}\]

Other definitions can be defined in a manner such as mosaic knots, for instance, connection points, contiguous, suitably connected. 
An \textit{$(m, n)$-mosaic} is an $m\times n$ matrix $M=(M_{ij})$ of tiles, with rows and columns indexed $0, 1, \cdots, m-1$ where each $(i,j)$-entry $M_{ij}$ is an element of $\mathbb{T}^{(u)}\cup \mathbb{T}^{(u)}_{M}$. The set of $(m, n)$-mosaics is denoted by $\mathbb{M}_{M}^{(m, n)}$. 
It $m=n$, then an $(n, n)$-mosaic is a \textit{$n$-mosaic} and its set is denoted by $\mathbb{M}_{M}^{(n)}$.

\begin{definition}
 A \textit{marked graph $(m, n)$-mosaic} is a $(m, n)$-mosaic in which all tiles are suitably connected. The set of all marked graph $(m, n)$-mosaic is the subset of $\mathbb{M}^{(m, n)}_{M}$, denoted by $\mathbb{K}^{(m, n)}_{M}$. 
If $m=n$, then it is called a \textit{marked graph $n$-mosaic} and its set is denoted by $\mathbb{K}^{(n)}_{M}$.
\end{definition}

\begin{example}
The marked graph diagrams $0_1$, $2_{1}^{1}$ and $6_{1}^{0,1}$ have the marked graph mosaics as follows. 
\[\scalebox{0.8}{\includegraphics{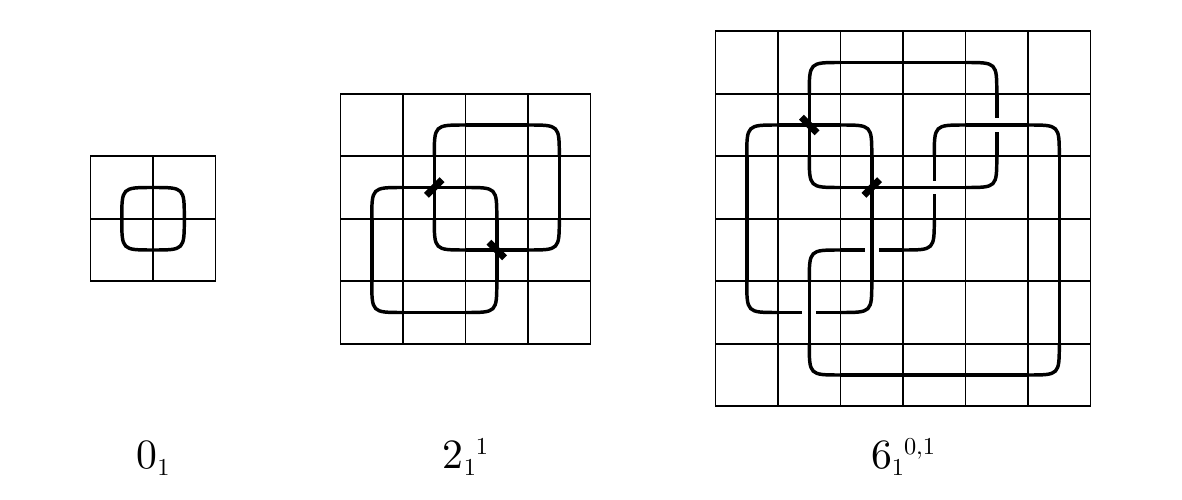}}\]
\end{example}

For the equivalence for marked graph mosaics, there are planar isotopy moves and Yoshikawa moves by using mosaic tiles in $\mathbb{T}^{(u)}\cup \mathbb{T}^{(u)}_{M}$. 
The mosaic moves for planar isotopy are the same $P_{1}, \cdots, P_{11}$ with knot mosaic moves and 4 additional moves $P_{8}', P_{9}', P_{10}', P_{11}'$ depicted as follows. 
\[\scalebox{0.8}{\includegraphics{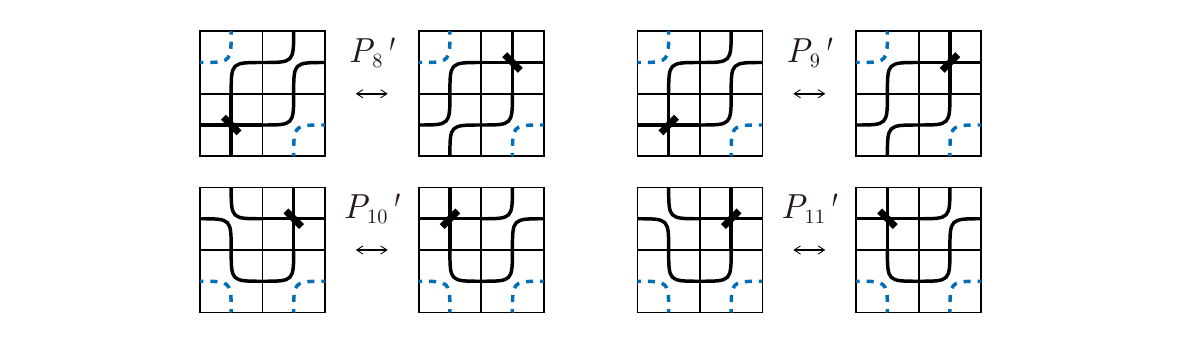}}\]
Yoshikawa moves $\Gamma_{1}, \Gamma_{2}, \Gamma_{3}$ are the same with Reidemeister moves I, II, III.
The mosaic moves for Yoshikawa moves $\Gamma_{4}, \cdots, \Gamma_{8}$ are as follows. 
\[\scalebox{0.8}{\includegraphics{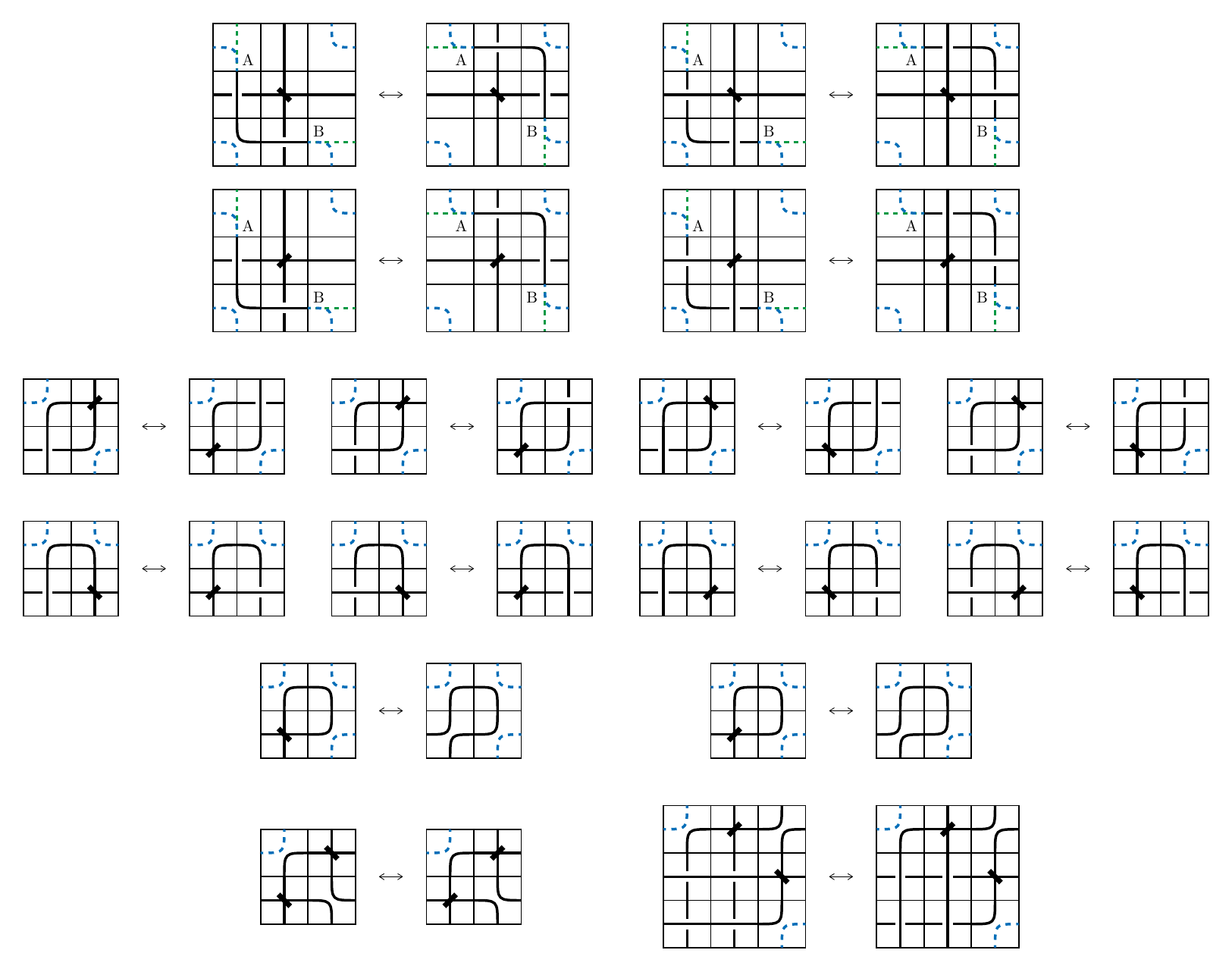}}\]


All marked graph mosaic moves are permutations on the set $\mathbb{M}^{(n)}_{M}$ of $n$-mosaics. Indeed, they are also in the group of all permutations of the set $\mathbb{K}^{(n)}_{M}$ of marked graph $n$-mosaics.

\begin{definition}
The \textit{ambient isotopy group $\mathbb{A}^{(n)}_{M}$} is the subgroup of the group of all permutations of the set $\mathbb{K}^{(n)}_{M}$ generated by all planar isotopy moves and all Yoshikawa moves.  
\end{definition}

Two marked graph $n$-mosaics $M$ and $M'$ are said to be \textit{of the same marked graph $n$-type}, denoted by 
  $M \overset{n}{\sim} M',$
  if there exists an element of $\mathbb{A}^{(n)}_{M}$ such that it transforms $M$ into $M'$.
Two marked graph $n$-mosaics $M$ and $M'$ are said to be \textit{of the same marked graph type} 
  if there exists a non-negative integer $k$ such that 
  $$i^{k}M \overset{n+k}{\sim} i^{k}M',$$
  where $i : \mathbb{M}^{(j)} \rightarrow \mathbb{M}^{(j+1)}$ is the mosaic injection by adding a row and a column consisting of only empty tiles.
Therefore, we can obtain the following result.  
  
\begin{theorem}
Let $M$ and $M'$ be two marked graph mosaics of two marked graphs $K$ and $K'$, respectively. 
Then $M$ and $M'$ are of the same marked graph mosaic type if and only if $K$ and $K'$ are equivalent.
\end{theorem}

For oriented surface-links, consider original oriented mosaic tiles in $\mathbb{T}^{(o)}$ (see in \cite{LK}) and add 4 oriented mosaic tiles with markers as follows. 
Then we can deal with oriented marked graph mosaics similar to oriented knot mosaics. 
\[\includegraphics{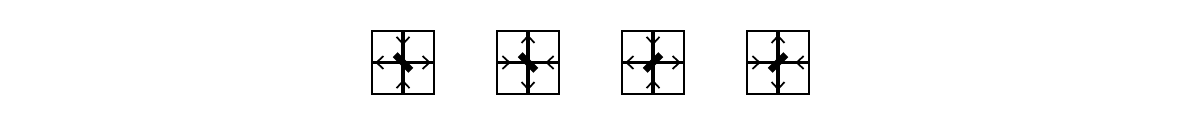}\]
The definition of suitably connected when an orientation is given also considers only cases where the orientation is well matched.
Therefore, the oriented marked graph mosaics can also follow the same flow.

\section{Mosaic numbers}

The marked graph diagram $8_1$ can reduce the size of its marked graph mosaic using mosaic moves. 
\[\scalebox{0.6}{\includegraphics{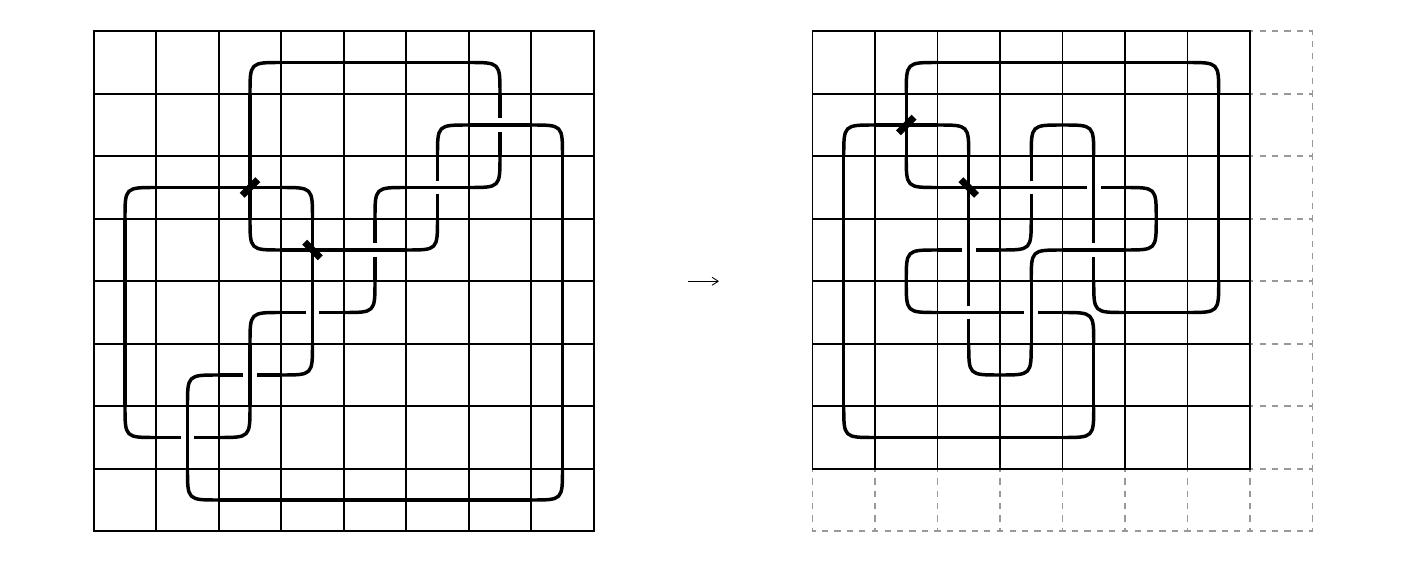}}\]

\begin{definition}
The \textit{mosaic number} of a marked graph diagram $K$, denoted by $m(K)$, is the smallest integer $n$ for which $K$ can be represented by a marked graph $n$-mosaic.
\end{definition} 

It is obvious that the smallest number of the mosaic size of a marked graph diagram is an invariant for surface-links. 

\begin{theorem}
The mosaic number of a marked graph diagram is an invariant for surface-links. 
\end{theorem}

It is obvious that the mosaic number of the standard sphere $0_1$ is $2$ and the mosaic numbers of both $2_{1}^{1}$ and $2_{1}^{-1}$ are $4$.

For finding the mosaic numbers, one can use twofold rule, introduced in \cite{OHLL}. 
For a given $(m, n)$-mosaic $D$, 
since there are exactly two ways to connect adjacent connection points in the boundary of $D$,
one can obtain exactly two marked graph $(m+2, n+2)$-mosaics $\widehat{D}^{1}$ and $\widehat{D}^{2}$,
where $D$ is suitably connected except the connection point of its boundary. 
The entry tiles of $D$ are called {\it inner tiles} of $\widehat{D}^{1}$ or $\widehat{D}^{2}$.
It is obvious that a crossing and a marked vertex must be located in the position of inner tiles for the suitably connected condition. 

\[\scalebox{0.8}{\includegraphics{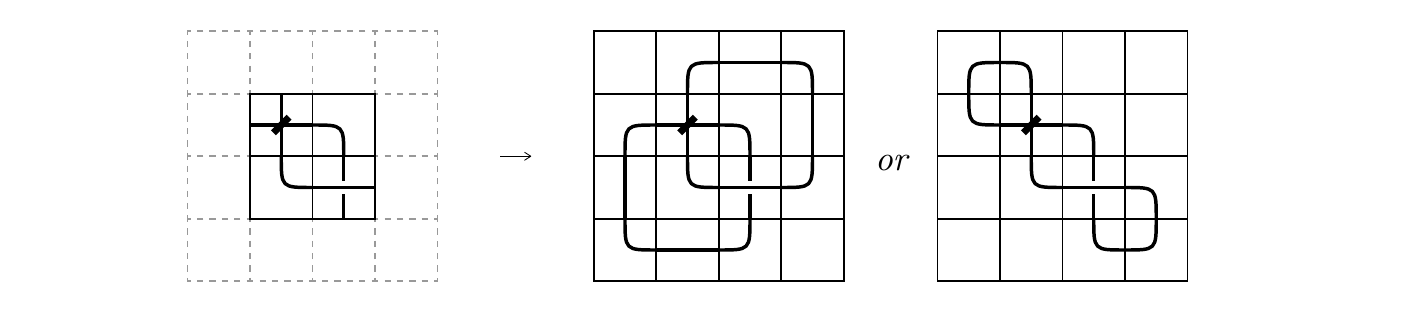}}\]

It is clear that if one of four inner corners has a crossing or a marked vertex and if one of two mosaics by the twofold rule makes a kink, then the crossing or the marked vertex can be removed by $\Gamma_{1}$ or $\Gamma_{6}, \Gamma_{6}'$, respectively.

\[\scalebox{0.8}{\includegraphics{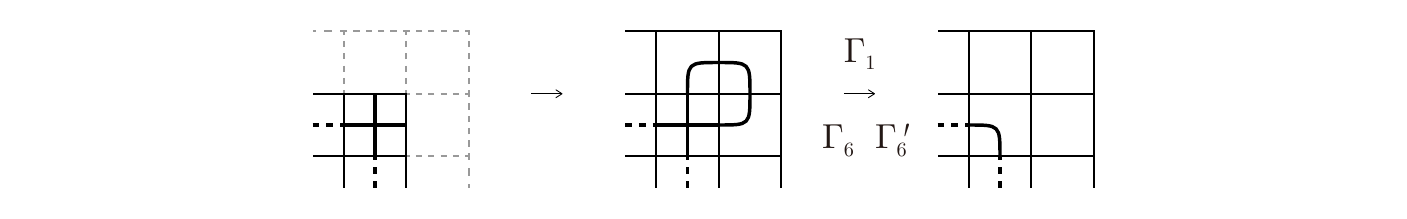}}\]

\begin{theorem}\label{Mnbr6}
Let $K$ be a marked graph $K$. 
If $\mathrm{ch}(K)\ge 7$, then $m(K)\geq6$ where $\mathrm{ch}(K)$ denotes the ch-index of $K$. 
\end{theorem}

\begin{proof}
Let $K$ be a marked graph whose ch-index is greater than or equal to $7$.
If $\mathrm{ch}(K)\geq 10$, then $m(K)\geq 6$ because the number of inner tiles of a $5$-mosaic diagram is $9$.
Similarly, it is easy to check that $m(K)\geq 5$ if $\mathrm{ch}(K)\geq 7$.

In the case that $\mathrm{ch}(K)=8$, we will show that $m(K)\ne 5$.
Suppose that $m(K)=5$, that is, there is a marked graph $5$-mosaic diagram $D$ of $K$ such that the ch-index of $D$ is $8$.
Since the number of inner tiles of $D$ is $9$, there are $9$ types for inner tiles.
All cases have at least $1$ row in the boundary of inner tiles, whose all mosaic tiles are crossings or marked vertices, as follows up to rotation.

\[\scalebox{0.8}{\includegraphics{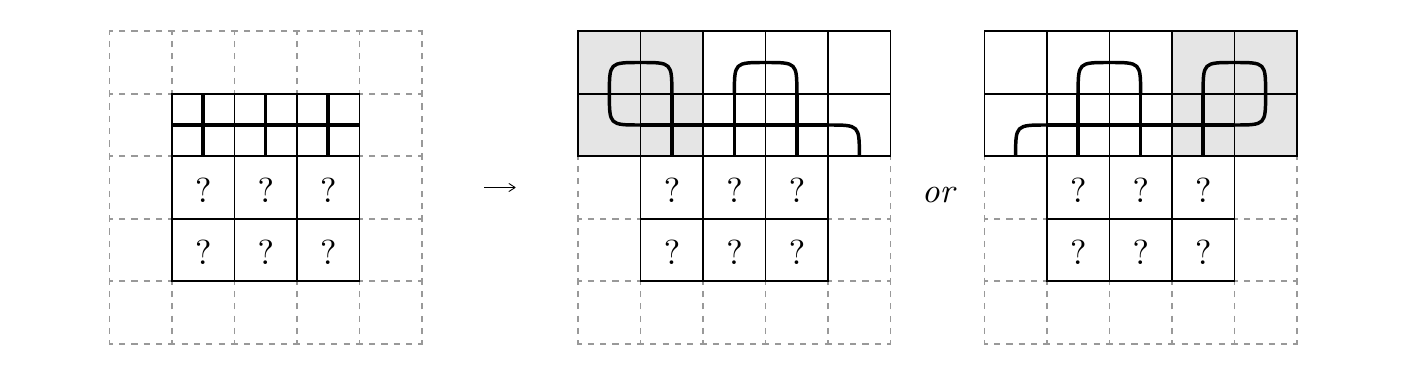}}\]

By applying the twofold rule, the resulting mosaics have always at least one kink. Therefore, one can remove the corresponding crossing or marked vertex. 
It contradicts that the ch-index is $8$. 
Hence, $m(K)\geq 6$.

Similar that $\mathrm{ch}(K)=7$, suppose that $m(K)=5$.
Let $D$ be a marked graph $5$-mosaic diagram of $K$ with ch-index $7$.
Then there are $36$ cases of its inner tiles and they have at least $1$ row as depicted above except $2$ cases. By applying the same argument of the case of $\mathrm{ch}(K)=8$, $34$ cases are contradictory. 
In the remaining $2$ cases, both have exactly two corners with no crossings and no marked vertices. Then for each cases, there are $4$ subcases as follows.
\[\scalebox{0.8}{\includegraphics{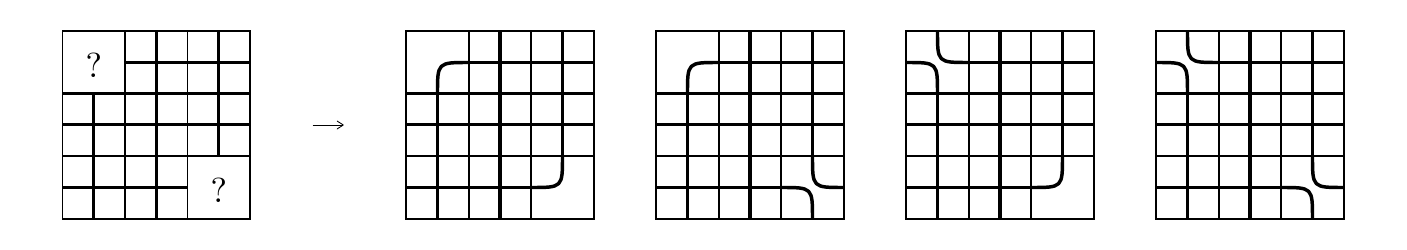}}\]
By the twofold rule, for each subcase, there two marked graph mosaics; one of them has always at least one kink. Since we can reduce the ch-index of $D$, it contradicts that the ch-index is $7$ and then $m(K)\geq 6$.
\[\scalebox{0.5}{\includegraphics{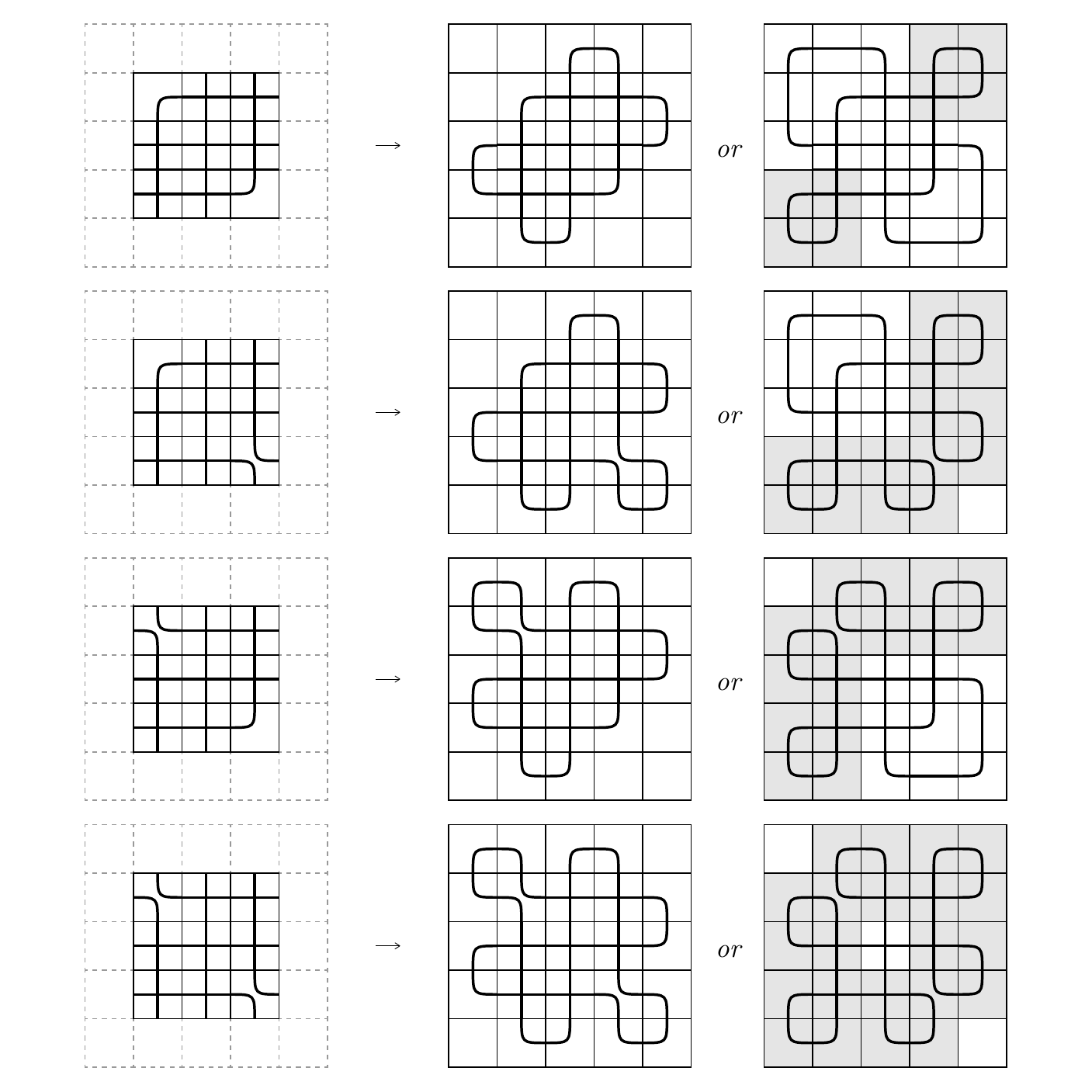}}\]
The remaining diagrams of $4$ subcase are the same shown as follows. 
\[\scalebox{0.8}{\includegraphics{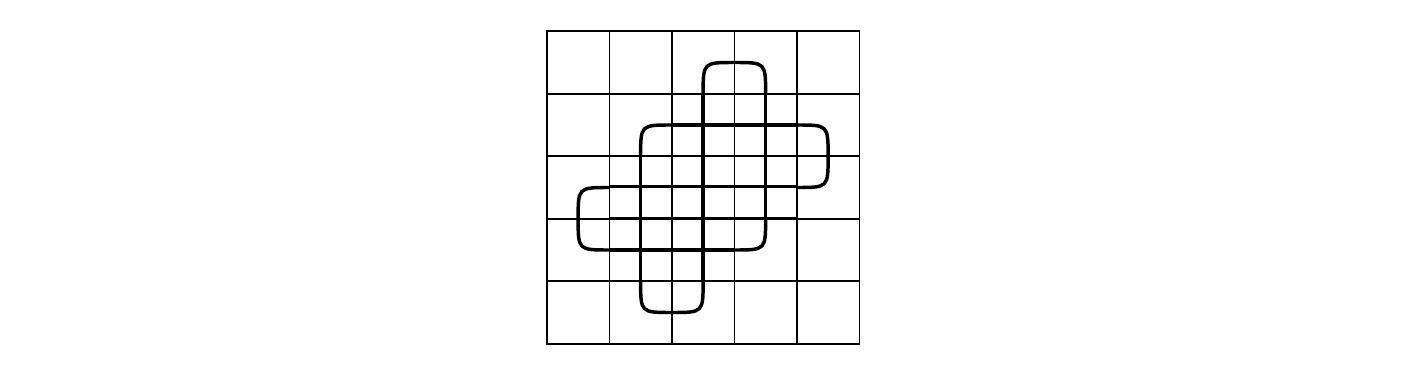}}\]
It has exactly one component. It contradicts that the number of components of $7_{1}^{0,-2}$ has two components. Hence, $m(K)\geq 6$.
\end{proof}

The following diagrams are marked graph mosaics of surface-links with ch-index $\leq$ 10. The size of some mosaic diagrams are $6$ as follows. By Theorem \ref{Mnbr6}, we know that their mosaic numbers are exactly $6$.
\[\scalebox{0.5}{\includegraphics{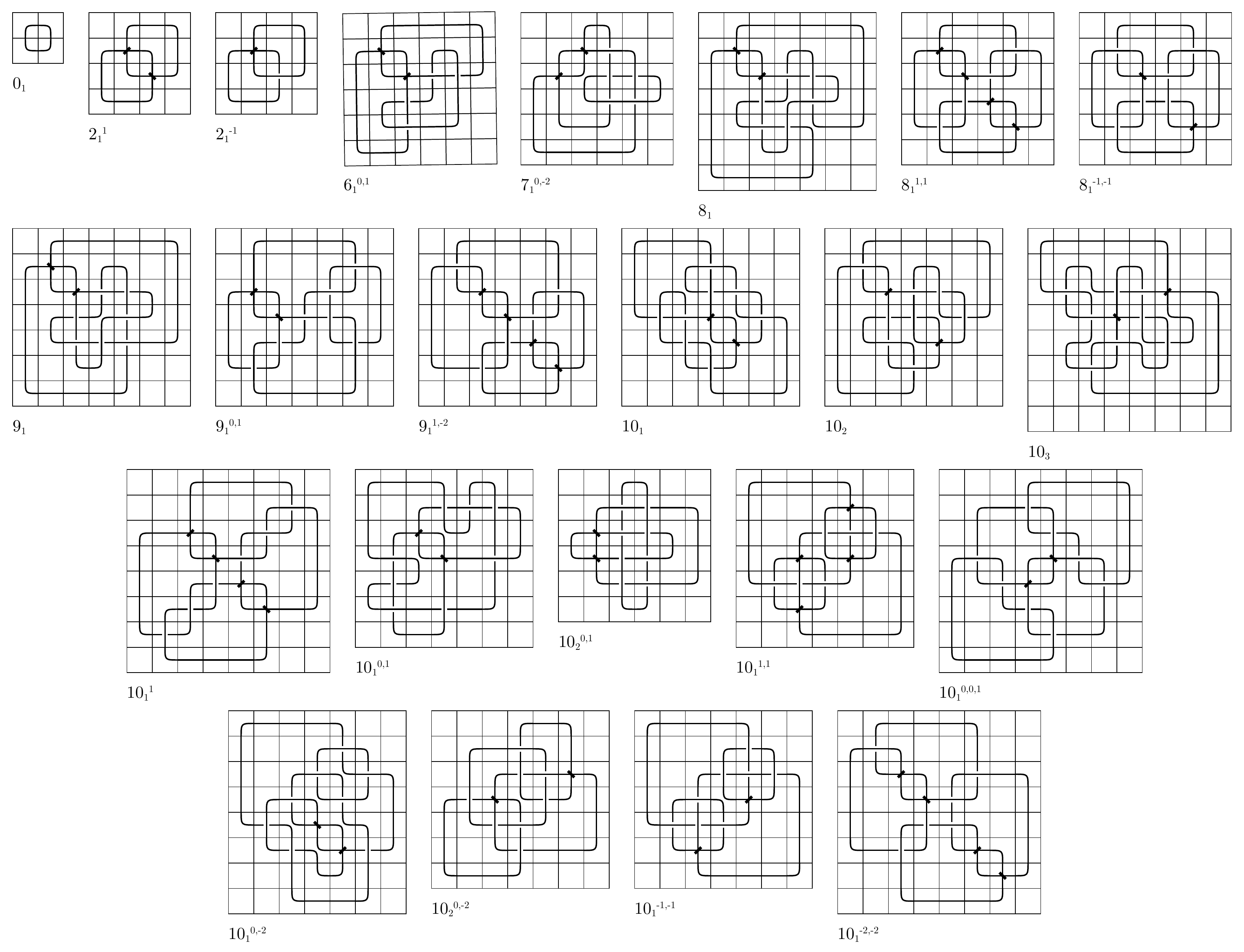}}\]


We conclude this section with a table of mosaic numbers for surface-links
of small ch-index.

\begin{center}
  \begin{tabular}[ ]{| c | c | }
    \hline 
    $K$ & $m(K)$   \\ 
    \hline \hline
    $0_{1}$                                     & $2$  \\ \hline 
    $2_{1}^{1}$, $2_{1}^{-1}$         & $4$ \\ \hline
    $6_{1}^{0,1}$                           & $5, 6$ \\ \hline 
    $7_{1}^{0,-2}$, $8_{1}^{1,1}$, $8_{1}^{-1,-1}$, $10_{2}^{0,1}$ & $6$ \\ \hline
    $8_{1}$, $9_{1}$, $9_{1}^{0,1}$, $9_{1}^{1,-2}$, $10_{1}$, $10_{2}$,  $10_{1}^{0,1}$, $10_{1}^{1,1}$, $10_{2}^{0,-2}$, $10_{1}^{-1,-1}$ & $6, 7$ \\ \hline
    $10_{3}$, $10_{1}^{1}$,  $10_{1}^{0,0,1}$, $10_{1}^{0,-2}$, $10_{1}^{-2,-2}$ & $6, 7, 8$\\ \hline
  \end{tabular}
\end{center}

\section{\large\textbf{Kei-Colored Mosaic Diagrams}}\label{K}

Recall that a \textit{kei} is a set $X$ with a binary operation $\tr$
satisfying the axioms
\begin{itemize}
\item[(i)] For all $x\in X$, $x\tr x=x$,
\item[(ii)] For all $x,y\in X$, we have $(x\tr y)\tr y=x$, and 
\item[(iii)] For all $x,y,z\in X$ we have $(x\tr y)\tr z=(x\tr z)\tr (y\tr z)$.
\end{itemize}
A map $f:X\to X'$ between kei is a \textit{kei homomorphism} if it satisfies
\[f(x\tr y)=f(x)\tr f(y)\]
for all $x,y\in X$. Kei are also called \textit{involutory quandles}; see 
\cite{EN} for more.

\begin{example}
Every group is a kei under the operation $x\ast y=yx^{-1}y$, called the 
\textit{core kei} of the group.
\end{example}

Every surface-link $L$ (including classical knots and links, which can be
regarded as trivial cobordisms) has a \textit{fundamental kei} $\mathcal{K}(L)$
whose presentation can be read from a diagram of the surface-link. More 
precisely, the fundamental kei of a surface-link has generators corresponding 
to \textit{sheets}, i.e., connected components of a marked graph diagram 
representing $L$ where we divide at classical undercrossings, together with 
relations at the crossings as shown (suggestively as mosaic tiles)
\[\includegraphics{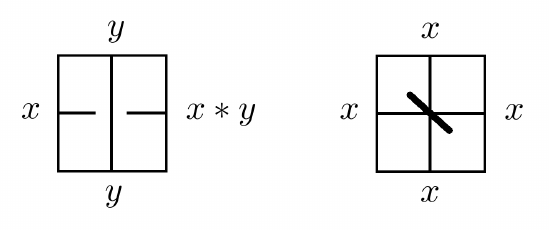}\]
The elements of the fundamental kei are then equivalence classes of kei
words in these generators modulo the equivalence relation generated by
the crossing relations and the kei axioms. The isomorphism class of the 
fundamental kei is a well-known invariant of unoriented surface-links.
 
Given a finite kei $X$, an assignment of elements of $X$ to the 
sheets of an oriented marked graph diagram (i.e., segments ending at 
undercrossing points or marked vertices) is a \textit{kei coloring} (also 
called an \textit{$X$-coloring}) of the diagram if it satisfies the crossing
condition pictured above at every crossing.

An $X$-coloring of a diagram $D$ of a surface-link $L$ defines and is defined 
by a unique element of the set of kei homomorphisms 
$\mathrm{Hom}(\mathcal{K}(L),X)$. This homset is an invariant of surface-links
for every finite kei $X$, from which useful computable invariants can be
extracted. The simplest example is the cardinality of the set, known as
the \textit{kei counting invariant}, denoted 
$\Phi_X^{\mathbb{Z}}(L)=|\mathrm{Hom}(\mathcal{K}(L),X)|$.

Generally speaking, any invariant of kei-colored diagrams (or equivalently,
homset elements) yields an invariant known as an \textit{enhancement} of the
kei counting invariant. Examples include the celebrated \textit{cocyle 
invariants} studied in \cite{CJKLS} and the more recent \textit{kei module 
invariants} introduced in \cite{JN}.

We will use mosaic diagrams to enhance the kei counting invariant in the 
following way. Let $L$ be a surface-link with mosaic diagram $D$ and let
$X$ be a finite kei. Assigning elements of $X$ (called ``kei colors'') 
to each of the arcs on the tiles in $D$ such that the colors match
at connection points and satisfy the kei coloring conditions at the 
crossings and marked vertices, we obtain an \textit{$X$-colored mosaic diagram}.
If we let $f\in\mathrm{Hom}(\mathcal{K}(L),X)$ be the homset element
represented by this coloring, we may denote the colored diagram by $D_f$.

\begin{definition}\label{def:k}
Let $L$ be a surface-link represented by a marked graph diagram $D$ 
and let $X$ be a finite kei. For each kei coloring
$f\in\mathrm{Hom}(\mathcal{K}(L),X)$ let us define the \textit{kei deficiency} 
of $D_f$ as the difference between the cardinality of the image
subkei of $f$ and the number of kei colors appearing in $D_f$. 
Let $\phi_f$ be the minimal kei deficiency over the set of minimal mosaic 
number diagrams $D_f$ representing $f$. Then the multiset 
\[\Phi_X^{\mathrm{Mos},M}(L)=\{\phi_f\ |\ f\in\mathrm{Hom}(\mathcal{K}(L),X)\}\]
is the \textit{mosaic deficiency enhancement multiset} of the kei homset 
invariant.
For ease of comparison we may also convert this to polynomial form by summing
over the multiset terms of the form $u^{\phi_f}$ to define the 
\textit{mosaic deficiency enhancement polynomial} 
\[\Phi_X^{\mathrm{Mos}}(L)=\sum_{f\in\mathrm{Hom}(\mathcal{K}(L),X)} u^{\phi_f}.\]
\end{definition}

Since there may be many distinct equivalent diagrams of $L$ with minimal 
mosaic number, to get an invariant we take for each coloring the minimal kei 
deficiency over the (finite) set
of all diagrams of $L$ with minimal mosaic number. Then by construction, the 
multiset of $\phi_f$-values forms an 
invariant of surface-links. From a given minimal-mosaic number diagram of $L$
we can obtain an upper bound on each of the $\phi_f$-values; to compute the 
invariant in general requires finding the complete set of minimal-mosaic number
diagrams of $L$, which can be computationally difficult. 

Let us order the set of polynomials with nonnegative integer coefficients  lexicographically by exponent. That is, to compare two polynomials we first compare their constant terms and in case of a tie, we use the linear term as a tiebreaker; if the constant and linear terms are equal, we use the quadratic term as a tiebreaker etc.
Then finding a new diagram which reduces the deficiency moves a coloring 
representative from a higher exponent into a lower exponent, yielding 
a smaller lexicographical position; hence it follows that any particular diagram yields
an upper bound on the invariant. 

To prove tightness of this
bound, one can check exhaustively (which we have not done in the Example below) that all other
mosaic diagrams with the same or lesser mosaic number of the link or surface-link in question have the same deficiencies for their colorings
representing the nontrivial homset elements.

\begin{remark}\label{rem:1}
We observe that we can similarly define deficiency enhancements using
crossing number or ch-index in place of mosaic number. Generally speaking,
on any diagram with nonzero deficiency we can perform Reidemeister II moves
to reveal ``missing'' colors in the image subkei. Since these moves increase 
ch-index without changing the mosaic number, we expect that these should be 
different invariants. 
\end{remark}

\begin{example}
Consider the surface-knot $10_1$ and the kei $\mathrm{Core}(\mathbb{Z}_5)$.
Our \texttt{python} computations show that $10_1$ has 25 colorings by the 
kei $\mathrm{Core}(\mathbb{Z}_5)$. These include five monochromatic colorings
which have deficiency zero and 20 nontrivial colorings, each of which is 
surjective with deficiency $1$ on this diagram, e.g.
\[\includegraphics{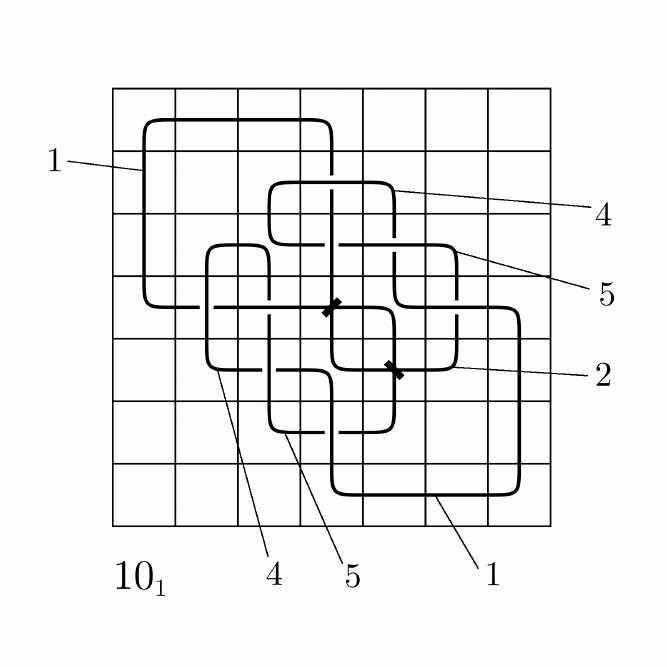}.\]
Then from this diagram we obtain an upper bound $5+20u$ on the kei deficiency 
polynomial.
\end{example}

We end this section by defining another easy-to-define but difficult-to-compute
invariant us surface-links using mosaics and kei. 

\begin{definition}
Let $L$ be a surface-link and $X$ a finite kei. For each 
$f\in\mathrm{Hom}(\mathcal{K}(L),X)$ and each positive integer $n\ge 2$,
let $\rho_f^n$ be the minimal kei deficiency value over all $n$-mosaic 
diagrams of $L$.
Then the sequence $\{\rho_f^n\}_{n=2}^{\infty}$ is the \textit{kei deficiency 
spectrum} for $f$, and as before we have an invariant multiset of such
spectra.
\end{definition}

\begin{remark}
We note that since classical knots can be understood as surface-links with an
empty set of marked vertices (i.e., trivial cobordisms between two copies
of the knot), the invariants defined in this section are also invariants
of classical knots and links.
\end{remark}

\section{\large\textbf{Questions}}\label{Q}

There remains much to be done on the topic of mosaic surface-links. Finding
efficient ways to prove tightness of bounds is of interest, as is
extending the quantum knot constructions in \cite{LK}.

Say a surface-link $L$ is \textit{$X$-deficiency heterogeneous} if it has at 
least two homset elements which require different minimal-mosaic number 
diagrams to realize their minimal $X$-deficiencies. Is there any such 
surface-link? For a given kei $X$, which is the smallest $ch$-index of
any link which is $X$-deficiency heterogeneous? For a fixed surface-link
$L$, for which finite kei $X$, if any, is $L$ $X$-deficiency heterogeneous?

A question raised by Seiichi Kamada at a talk on this topic while this
paper was in preparation is whether the ordering of surface-links
by ch-number agrees with that induced by mosaic number -- e.g.,
does there exist a surface-link whose minimal ch-diagram has greater
mosaic number than its minimal mosaic diagram. As mentioned in Remark 
\ref{rem:1}, since there are moves which change the ch-index without changing 
the mosaic number, it is not clear what is the relationship between these two
notations of complexity of surface-links.

\bibliographystyle{abbrv}
\bibliography{sc-sn}{}

\bigskip

\noindent
\textsc{Nonlinear Dynamics and Mathematical Application Center \\
Kyungpook National University \\
Daegu, 41566, Republic of Korea} 

\bigskip

\noindent
\textsc{Department of Mathematical Sciences \\
Claremont McKenna College \\
850 Columbia Ave. \\
Claremont, CA 91711 USA}

\end{document}